\documentclass[pdflatex,sn-mathphys-num]{sn-jnl}

\usepackage{amsmath,amssymb,amsthm,mathtools,mathrsfs}
\usepackage{graphicx}%
\usepackage{multirow}%
\usepackage{amsmath,amssymb,amsfonts}%
\usepackage{amsthm}%
\usepackage{mathrsfs}%
\usepackage[title]{appendix}%
\usepackage{xcolor}%
\usepackage{textcomp}%
\usepackage{manyfoot}%
\usepackage{booktabs}%
\usepackage{algorithm}%
\usepackage{algorithmicx}%
\usepackage{algpseudocode}%
\usepackage{listings}%
\usepackage{xcolor}
\usepackage{graphicx}
\usepackage{booktabs}
\usepackage{hyperref}
\hypersetup{
  colorlinks=true,
  linkcolor=blue!55!black,
  citecolor=blue!55!black,
  urlcolor=blue!55!black
}

\theoremstyle{plain}
\newtheorem{theorem}{Theorem}[section]
\newtheorem{proposition}[theorem]{Proposition}
\newtheorem{lemma}[theorem]{Lemma}

\theoremstyle{definition}
\newtheorem{definition}[theorem]{Definition}
\newtheorem{assumption}[theorem]{Assumption}
\newtheorem{remark}[theorem]{Remark}

\newcommand{\R}{\mathbb{R}}
\newcommand{\interior}[1]{\mathring{#1}}
\DeclareMathOperator{\sgn}{sgn}
\DeclareMathOperator{\dist}{dist}

\begin{document}
\title{\bf Local Time and Killed Resolvents in Reflected Optimal Stopping with a Max Payoff}

\author*[1]{\fnm{Louis Shuo} \sur{Wang}}
\email{wang.s41@northeastern.edu}
\equalcont{These authors contributed equally to this work.}

\author*[2]{\fnm{Ye} \sur{Liang}}\email{ye-liang@uiowa.edu}
\equalcont{These authors contributed equally to this work.}

\affil[1]{\orgdiv{Department of Mathematics},
\orgname{Northeastern University},
\orgaddress{
\city{Boston},
\state{MA},
\postcode{02115},
\country{USA}
}}

\affil[2]{\orgdiv{College of Engineering},
\orgname{The University of Iowa},
\orgaddress{
\city{Iowa City}, 
\postcode{52242}, 
\state{IA},
\country{USA}}}

\abstract{
We study infinite-horizon optimal stopping for normally reflected
two-dimensional diffusions in the positive quadrant with max payoff
\(G(x_1,x_2)=x_1\vee\alpha x_2\). The non-smooth payoff produces a singular
stopping-gain measure on the kink set \(\Delta=\{x_1=\alpha x_2\}\). We prove
$\displaystyle   \Gamma^\Delta(dx)
   =
   -\frac{n^\top a(x)n}{2\sqrt{1+\alpha^2}}\,\sigma_\Delta(dx)$,
with $n=(1,-\alpha)$,
so the diagonal component is non-positive and strictly negative under local
ellipticity. This implies that every interior kink point lies in the
continuation region. We further show that the correct value representation uses
the resolvent killed at first entry into the stopping set,
$\displaystyle V=G-R_r^{\mathcal C}\Gamma$,
and give a closed-form reflected Brownian counter-example showing that the
unrestricted reflected resolvent is generally wrong. A reflected Brownian
benchmark and numerical experiments illustrate the local-time, resolvent-gap,
and diagonal-avoidance mechanisms.
}
\maketitle

\section{Introduction}
\label{sec:introduction}

\subsection{The problem}
\label{sec:intro-problem}
We consider infinite-horizon optimal stopping for a normally reflected
diffusion $X=(X^1,X^2)$ in the closed quadrant $\R^2_+=[0,\infty)^2$. A decision
maker chooses a stopping time $\tau$ to maximise the discounted reward net of a
running cost,
\begin{equation}
\label{eq:intro-value}
   V(x)
   =
   \sup_{\tau}
   \mathbb E_x\left[
      e^{-r\tau}G(X_\tau)
      -\int_0^\tau e^{-rs}c(X_s)\,ds
   \right],
   \qquad
   G(x_1,x_2)=x_1\vee\alpha x_2,
\end{equation}
with $\alpha>0$, discount $r>0$, and non-negative running cost $c$. Two features
make \eqref{eq:intro-value} a genuinely two-dimensional, non-smooth problem.
First, $X$ is confined to the quadrant by normal reflection at the two
coordinate axes, so the generator is supplemented by Neumann boundary behaviour,
and the corner needs separate care. Second, the max-type reward is only
piecewise affine: it has a convex Lipschitz kink along the diagonal
$\Delta=\{x_1=\alpha x_2\}$, across which its distributional second derivatives
concentrate as a surface measure. Both features break the smoothness that the
textbook free-boundary calculus tacitly assumes.

Max-type rewards $x_1\vee\alpha x_2$ are the canonical payoff for American
options on the maximum of several assets and for related exchange and
dual-strike contracts; the qualitative structure of the exercise regions was
analysed in \cite{broadie1997valuation,villeneuve1999exercise,bayer2020pricing,battauz2015real,wang2025analysis,dai2005american,detemple2014optimal,wang2025analysis1,soner2025stopping,whitehead2012bias}. Reflected dynamics in the orthant arise whenever the
state is a regulated or constrained process, as in queueing, inventory, and
economic models with reflecting barriers; the probabilistic foundations trace to
the Skorokhod-problem analyses of Tanaka \cite{tanaka1979stochastic}, Lions and Sznitman
\cite{lions1984stochastic}, Saisho \cite{saisho1987stochastic}, and Harrison and Reiman
\cite{harrison1981reflected}, with oblique reflection on non-smooth domains treated in
\cite{liu2025bidirectional,dupuis1993sdes,hino2021pathwise,lundstrom2019stochastic,liang2025global,li2026wong,lipshutz2018directional,deblassie1996brownian,duan2019white,wang2026algebraic,leimkuhler2023simplest,cai2026optimal,wang2026damage}.

\subsection{What is unconditional and what is conditional}
\label{sec:intro-tiers}
The analysis separates cleanly into two tiers, and keeping them apart is one of
the organising aims of the paper.

\emph{Tier I (unconditional, within the standing reflected-diffusion
hypotheses).} The decomposition of the stopping-gain object
$\Gamma=c+rG-\mathcal LG$ into an absolutely continuous part and an explicit
signed diagonal surface measure; the strict negativity of that diagonal measure
under local ellipticity; the resulting theorem that no interior diagonal point
is ever a stopping point; the failure of the unrestricted reflected resolvent;
and the killed-resolvent representation of the value once the first-entry time
is known to be optimal.

\emph{Tier II (conditional free-boundary geometry).} The epigraph
representation of the stopping set, monotonicity of the boundary, boundary trace
conditions, and candidate verification. These require structural monotonicity
and regularity hypotheses that must be checked per model; the paper states them
explicitly rather than smuggling them in.

The theory of optimal stopping and free-boundary problems is now classical and is surveyed in the monograph of Peskir and Shiryaev \cite{peskir2006optimal}; its probabilistic backbone---Snell envelopes, dynamic programming, and the connection to variational inequalities---is equally well established \cite{el2006aspects,bensoussan2011applications,yu2026pattern,friedman1982variational,jacka2019compensator,martyr2016finite,wang2026breakdown,christensen2021class,ma2021dynamic,yu2026rigorous,ankirchner2019verification,lv2026robust,yu2026beyond,jeon2023variational}. 
On the analytic side, the obstacle problem has a mature regularity theory, and the structure of the free boundary has been studied extensively since Caffarelli \cite{caffarelli1998obstacle} and in many later works \cite{caffarelli2008regularity,wang2025multi,figalli2024constraint,focardi2022local,huang2024regularity,gao2022rolling,du2026free,aleksanyan2024quantitative,eberle2023structure}. 
Viscosity solutions provide the standard weak framework in which value functions can be characterised without a priori smoothness \cite{crandall1992user,yu2026from,soner2024viscosity,cheng2025viscosity}, while the distributional or measure-valued reformulation of the one-dimensional problem was developed by Lamberton and Zervos \cite{lamberton2013optimal}.

This split is the response to a basic difficulty. In the optimal stopping
literature for multidimensional and time-inhomogeneous problems, free-boundary
regularity has been obtained only under problem-specific hypotheses and by
substantial work; see De Angelis \cite{de2015note}, Peskir \cite{peskir2019continuity},
De Angelis and Stabile \cite{de2019lipschitz}, and De Angelis and Peskir
\cite{de2020global}. We therefore do not present free-boundary regularity
as automatic. What we do present unconditionally is the measure-theoretic and
potential-theoretic core, which is where the concrete and correctable
mathematics lies.

\subsection{Three corrections, made precise}
\label{sec:intro-corrections}
Heuristic treatments of problems like \eqref{eq:intro-value} commonly perform
three steps that are not justified as usually stated. We correct each.

\emph{(i) The stopping gain is a signed measure, not a function.} Because $G$ has
a convex kink on $\Delta$, the object $\Gamma=c+rG-\mathcal LG$ is not an
ordinary function: its singular part is a surface measure on $\Delta$ generated
by the local time of $Y=X^1-\alpha X^2$ at zero. We compute it in closed form by
Tanaka's formula and the occupation-density/co-area identity
(Theorem~\ref{thm:stopping-gain-measure}),
\[
   \Gamma^\Delta(dx)
   =-\frac{q(x)}{2\sqrt{1+\alpha^2}}\,\sigma_\Delta(dx),
   \qquad q=n^\top a\,n\ge0,
\]
which is non-positive as a measure and strictly negative where $a$ is
non-degenerate. A pointwise condition such as ``$\Gamma\ge0$ on the stopping
set'' is therefore ill-posed wherever the stopping set meets $\Delta$ in
positive length, unless the diagonal component is separated out. As an immediate
and concrete payoff, the strict sign forces every interior diagonal point into
the continuation region (Theorem~\ref{thm:no-stopping-kink}).

\emph{(ii) The potential must be killed at the stopping time.} Applying
It\^o--Tanaka up to the optimal time $\tau_{\mathcal D}$ yields
$\displaystyle   V(x)=G(x)-R_r^{\mathcal C}\Gamma(x)$,
where $R_r^{\mathcal C}$ is the resolvent \emph{killed} on entry into the
stopping set $\mathcal D$ (Theorem~\ref{thm:killed-resolvent-representation}). It
is generally not equal to the unrestricted reflected resolvent
$R_r^{\mathrm R}(\Gamma\mathbf 1_{\mathcal C})$, because the reflected process
continued past $\tau_{\mathcal D}$ may re-enter the continuation region. We make
this quantitative with a completely solved one-dimensional counter-example
(Proposition~\ref{prop:unrestricted-fails}).

\emph{(iii) Regularity is not a corollary of being a viscosity supersolution.} A
continuous viscosity supersolution of a linear elliptic inequality does not by
itself carry the Sobolev or measure structure needed for an
It\^o--Krylov--Tanaka argument \cite{krylov1987approach}. We separate the verification
principle (Theorem~\ref{thm:verification-measure}) from any regularity claim and
list its hypotheses explicitly.

\subsection{A verified benchmark}
\label{sec:intro-benchmark}
To show that the framework is operational and non-empty, not merely a list of
hypotheses, we verify it in full for a reflected Brownian model with constant
drift and constant uniformly elliptic covariance
(Section~\ref{sec:benchmark}). There the Skorokhod map is explicit, all standing
assumptions are checked, the value function is globally Lipschitz, the diagonal
stopping-gain measure is a strictly negative constant multiple of arc-length, the
no-stopping-on-the-diagonal theorem applies, and the killed-resolvent
representation holds at the optimal first-entry time.

\subsection{Numerical experiments}
\label{sec:intro-numerics}
The numerical part of the paper is designed to test the stochastic mechanisms
rather than to approximate a particular financial product. First, Monte Carlo
occupation-density estimates recover the $\sqrt t$ scaling of the local time at
the diagonal kink. Second, simulation of a one-dimensional reflected Brownian
motion reproduces the exact gap between the killed and unrestricted resolvents,
isolating the post-stopping re-entry mechanism. Third, a finite-difference
obstacle solver for the reflected Brownian benchmark visualises the positive
continuation advantage along the diagonal. These experiments support the
probabilistic interpretation of the theorems and help separate the
unconditional mechanism from the conditional free-boundary layer.

\subsection{Contributions}
\label{sec:intro-contributions}
The paper makes four contributions. First, it proves that the stopping gain
$c+rG-\mathcal LG$ is a signed measure and computes its diagonal surface
component explicitly (Theorem~\ref{thm:stopping-gain-measure}). Second, it
proves that strict negativity of this surface component forces every interior
point of the kink set into the continuation region
(Theorem~\ref{thm:no-stopping-kink}). Third, it proves that the value
representation uses the killed resolvent
(Theorem~\ref{thm:killed-resolvent-representation}) and gives a closed-form
reflected Brownian counter-example showing that the unrestricted resolvent is
wrong (Proposition~\ref{prop:unrestricted-fails}). Fourth, it verifies the
measure and killed-resolvent framework in a reflected Brownian benchmark
(Theorem~\ref{thm:benchmark}) and uses simulations to illustrate the
local-time, resolvent-gap, and diagonal-avoidance mechanisms
(Section~\ref{sec:numerics}).

\subsection{Organisation}
\label{sec:intro-Organisation}
Section~\ref{sec:model} fixes the reflected stopping problem, the obstacle
convention, and the verification principle.
Section~\ref{sec:measure-stopping-gain} computes the measure-valued stopping
gain and proves that the diagonal is never a stopping set.
Section~\ref{sec:why-killed} shows by counter-example why the unrestricted
resolvent is wrong. Section~\ref{sec:killed-resolvent} proves the
killed-resolvent representation and the boundary trace condition.
Section~\ref{sec:conditional-geometry} develops the conditional free-boundary
geometry. Section~\ref{sec:benchmark} verifies the framework in the reflected
Brownian benchmark. Section~\ref{sec:numerics} reports the Monte Carlo and
finite-difference experiments. Section~\ref{sec:discussion} collects limitations
and open problems.

\section{Reflected stopping problem and obstacle convention}
\label{sec:model}
This section fixes the probabilistic framework, the sign convention for the
obstacle problem, and the verification principle. The formulation is conditional
where it must be: the reflected diffusion is assumed to have the required Markov
and stability properties, and the verification theorem is stated under explicit
measure-superharmonicity hypotheses.

\subsection{Reflected diffusion and standing assumptions}
\label{sec:reflected-diffusion-standing}
Write $\R^2_+:=[0,\infty)^2$ and $\interior{\R}^2_+:=(0,\infty)^2$. Let
$W=(W^1,W^2)$ be a two-dimensional Brownian motion on a filtered probability
space $(\Omega,\mathcal F,(\mathcal F_t)_{t\ge0},\mathbb P)$ satisfying the usual
conditions. For each $x\in\R^2_+$ consider the normally reflected diffusion
\begin{equation}
\label{eq:reflected-sde}
   X_t
   =
   x
   +\int_0^t\mu(X_s)\,ds
   +\int_0^t\sigma(X_s)\,dW_s
   +L_t,
   \qquad t\ge0,
\end{equation}
with reflection process $L=(L^1,L^2)$ acting normally on the axes:
\begin{equation}
\label{eq:normal-reflection}
   L^i \text{ continuous, non-decreasing},\qquad
   L^i_0=0,\qquad
   \int_0^\infty \mathbf 1_{\{X_t^i>0\}}\,dL_t^i=0,
   \quad i=1,2.
\end{equation}
Let $a(x):=\sigma(x)\sigma(x)^\top$. The interior generator is
\[
   \mathcal L f(x)
   =
   \sum_{i=1}^2\mu_i(x)\partial_i f(x)
   +
   \frac12\sum_{i,j=1}^2a_{ij}(x)\partial_{ij}f(x),
   \qquad x\in\interior{\R}^2_+.
\]

\begin{assumption}[Standing assumptions]
\label{ass:reflected-diffusion}
\hfill
\begin{enumerate}
\item[\textup{(A1)}] $\mu:\R^2_+\to\R^2$ and $\sigma:\R^2_+\to\R^{2\times2}$ are
locally Lipschitz with at most linear growth.
\item[\textup{(A2)}] $a=\sigma\sigma^\top$ is locally uniformly elliptic in
$\R^2_+$: for every compact $K\subset\R^2_+$ there is $\lambda_K>0$ with
$\xi^\top a(x)\xi\ge\lambda_K|\xi|^2$ for $x\in K$, $\xi\in\R^2$.
\item[\textup{(A3)}] For each $x$, \eqref{eq:reflected-sde}--\eqref{eq:normal-reflection}
has a unique strong solution, and $(X^x)_x$ is strong Markov.
\item[\textup{(A4)}] For every $T>0$, compact $K\subset\R^2_+$, and some
$q\ge1$, $\lim_{y\to x}\mathbb E[\sup_{0\le t\le T}|X_t^y-X_t^x|^q]=0$ uniformly
on $K$.
\item[\textup{(A5)}] $G(x_1,x_2)=x_1\vee\alpha x_2$ with $\alpha>0$, and
$c:\R^2_+\to[0,\infty)$ is continuous.
\item[\textup{(A6)}] The value \eqref{eq:value-function} is finite with
polynomial growth: $0\le V(x)\le C_V(1+|x|^p)$ for some $C_V,p$, and the payoff
terms below are integrable for the stopping times in play.
\end{enumerate}
\end{assumption}

\begin{remark}[Reflected SDE hypotheses are not free]
\label{rem:reflected-sde-background}
Existence, uniqueness, and stability for reflected SDEs are delicate properties
of the Skorokhod problem, the domain, the reflection field, and the coefficients
\cite{tanaka1979stochastic,lions1984stochastic,saisho1987stochastic,harrison1981reflected,dupuis1993sdes};
they are part of the standing hypotheses, not consequences of writing down
\eqref{eq:reflected-sde}. Section~\ref{sec:benchmark} verifies them explicitly in
a concrete model.
\end{remark}

\begin{lemma}[Lyapunov condition for discounted integrability]
\label{lem:lyapunov-integrability}
Suppose there is $\Psi\in C^2(\R^2_+)$, $\Psi\ge1$, with $\partial_i\Psi=0$ on
$\{x_i=0\}$, and constants $K\ge0$, $\lambda<r$, $C_\Psi>0$ such that
$\mathcal L\Psi\le\lambda\Psi+K$ on $\interior{\R}^2_+$ and
$G^++c\le C_\Psi\Psi$. Then
$\mathbb E_x[\int_0^\infty e^{-rs}c(X_s)\,ds]<\infty$ and
$\sup_{t\ge0}\mathbb E_x[e^{-rt}G(X_t)^+]<\infty$.
\end{lemma}
\begin{proof}
Apply It\^o's formula with reflection to $e^{-rt}\Psi(X_t)$ and localise. The
Neumann condition on $\Psi$ removes the reflection term, since $dL^i$ is carried
by $\{X^i=0\}$. With $\phi(t):=\mathbb E_x[e^{-rt}\Psi(X_t)]$, the bound
$\mathcal L\Psi-r\Psi\le-(r-\lambda)\Psi+K$ gives
$\phi'(t)\le-(r-\lambda)\phi(t)+Ke^{-rt}$. Multiplying by the integrating factor
$e^{(r-\lambda)t}$ and integrating from $0$ to $t$,
\[
   \mathbb E_x[e^{-rt}\Psi(X_t)]
   \le e^{-(r-\lambda)t}\Psi(x)
   +K\int_0^t e^{-(r-\lambda)(t-s)}e^{-rs}\,ds .
\]
Since $r-\lambda>0$ and $r>0$, the right side is bounded uniformly in $t$, so the
discounted moments of $\Psi$ are finite; Tonelli gives finiteness of the
discounted cost integral, and $G^++c\le C_\Psi\Psi$ transfers the bounds.
\end{proof}

\subsection{Value function and dynamic programming}
\label{sec:value-dpp}
Let $\mathcal T$ be the set of $(\mathcal F_t)$-stopping times in $[0,\infty]$,
with $e^{-r\tau}G(X_\tau):=0$ on $\{\tau=\infty\}$. Set
\begin{equation}
\label{eq:value-function}
   V(x)
   =
   \sup_{\tau\in\mathcal T}
   \mathbb E_x\left[
      e^{-r\tau}G(X_\tau)
      -\int_0^\tau e^{-rs}c(X_s)\,ds
   \right],
\end{equation}
\begin{equation}
\label{eq:stopping-continuation}
   \mathcal D:=\{V=G\},
   \qquad
   \mathcal C:=\{V>G\}.
\end{equation}
By taking $\tau=0$, $V\ge G$, so $\mathcal D$ and $\mathcal C$ partition
$\R^2_+$. We write $H:=V-G\ge0$.

\begin{proposition}[Dynamic programming]
\label{prop:dpp}
Under Assumption~\ref{ass:reflected-diffusion}, for every bounded stopping time
$\theta$,
\[
   V(x)
   =
   \sup_{\tau\in\mathcal T}
   \mathbb E_x\Big[
   -\!\int_0^{\tau\wedge\theta}\!e^{-rs}c(X_s)\,ds
   +e^{-r\tau}G(X_\tau)\mathbf 1_{\{\tau\le\theta\}}
   +e^{-r\theta}V(X_\theta)\mathbf 1_{\{\tau>\theta\}}
   \Big].
\]
\end{proposition}
\begin{proof}
The standard dynamic programming principle for an optimal stopping problem driven
by a strong Markov process \cite{el2006aspects,peskir2006optimal}, using the
strong Markov property at $\theta$, concatenation of stopping rules, and the
integrability in Assumption~\ref{ass:reflected-diffusion}.
\end{proof}

\begin{lemma}[Lower semicontinuity]
\label{lem:value-lsc}
If, in addition, the uniform integrability needed below holds locally in the
initial state, then $V$ is lower semicontinuous; if $V$ is continuous, then
$\mathcal D$ is closed and $\mathcal C$ is open.
\end{lemma}
\begin{proof}
Fix $x\in \mathbb{R}_+^2$ and $\varepsilon>0$, and pick a bounded $\tau\le T$ that is
$\varepsilon$-optimal at $x$. Couple $X^x$ and $X^y$ on the same space with the
same $\tau$. By Assumption~\ref{ass:reflected-diffusion}\,(A4), continuity of $G$
and $c$, and the stated uniform integrability, the payoff at $y$ converges to
that at $x$; since $\tau$ is admissible from $y$,
$\liminf_{y\to x}V(y)\ge V(x)-\varepsilon$. Let $\varepsilon\downarrow0$. The
last claim is immediate from continuity of $H=V-G$.
\end{proof}

\subsection{Obstacle problem and reflected viscosity solutions}
\label{sec:obstacle-viscosity}
The obstacle problem is
\begin{equation}
\label{eq:obstacle-max}
\ \max\{(\mathcal L-r)V-c,\;G-V\}=0
   \quad\text{in }\interior{\R}^2_+
\end{equation}
with reflection conditions $\partial_iV=0$ on $\{x_i=0\}$, understood in the
relaxed sense below. For a test function $\varphi$ put
$\displaystyle    F[\varphi](x):=\max\{(\mathcal L-r)\varphi(x)-c(x),\;G(x)-\varphi(x)\}$.
At a contact point $u(x)=\varphi(x)$, the second entry equals $G(x)-u(x)$.

\begin{definition}[Reflected viscosity solution]
\label{def:reflected-viscosity}
Let $u$ be continuous with polynomial growth.
\begin{enumerate}
    \item $u$ is a reflected viscosity subsolution whenever $\varphi\in C^2(\mathbb{R}_+^2)$ and $u-\varphi$ has a local maximum at
$x\in \mathbb{R}_+^2$ with $u(x)=\varphi(x)$: if $x\in\interior{\R}^2_+$ then $F[\varphi](x)\le0$;
if $x_i=0$ then $\min\{F[\varphi](x),\partial_i\varphi(x)\}\le0$.
\item $u$ is a reflected viscosity supersolution whenever $u-\varphi$ has a local minimum at $x\in \mathbb{R}_+^2$ with
$u(x)=\varphi(x)$: if $x\in\interior{\R}^2_+$ then $F[\varphi](x)\ge0$; if
$x_i=0$ then $\max\{F[\varphi](x),\partial_i\varphi(x)\}\ge0$. 
\item $u$ is a reflected viscosity solution if it is both a reflected viscosity subsolution and a reflected viscosity supersolution.
\end{enumerate}
\end{definition}

\begin{proposition}[Value is a reflected viscosity solution]
\label{prop:value-viscosity}
Under the dynamic programming principle and the standing integrability, $V$ is a
reflected viscosity solution of \eqref{eq:obstacle-max} with the reflection
conditions \cite{crandall1992user,peskir2006optimal}.
\end{proposition}
\begin{proof}
The standard viscosity argument from the dynamic programming principle. At an
interior upper contact, the principle over a short horizon plus It\^o's formula
for $\varphi(X)$ gives $F[\varphi](x)\le0$; at a lower contact the usual
contradiction gives $F[\varphi](x)\ge0$. At a boundary point the reflection term
$\int_0^t\partial_i\varphi(X_s)\,dL^i_s$ produces the relaxed Neumann
inequalities.
\end{proof}

\subsection{Verification under measure superharmonicity}
\label{sec:verification}
Let $\mathcal O\subset\interior{\R}^2_+$ be open and
$\tau_{\mathcal O^c}:=\inf\{t:X_t\notin\mathcal O\}$.

\begin{definition}[Generalised It\^o class]
\label{def:generalised-ito-class}
A continuous $u$ belongs to the generalised It\^o class for $X$ if, whenever
$(\mathcal L-r)u-c$ is represented in the interior by a signed Radon measure
$\mu_u$, there is a signed continuous additive functional $A^{\mu_u}$ with, after
localisation,
\[
   e^{-rt}u(X_t)
   =u(x)+\int_0^t e^{-rs}c(X_s)\,ds+\int_0^t e^{-rs}\,dA_s^{\mu_u}+M_t
   +(\text{reflection}),
\]
$M$ a local martingale; the reflection terms vanish if $u$ satisfies the Neumann
condition in the trace sense \cite{revuz2013continuous,fukushima2011dirichlet}.
\end{definition}

\begin{theorem}[Verification]
\label{thm:verification-measure}
Let $u$ be continuous with polynomial growth and assume:
\textup{(V1)} $u\ge G$;
\textup{(V2)} $u=G$ on $\R^2_+\setminus\mathcal O$;
\textup{(V3)} $u\in W^{2,p}_{\mathrm{loc}}(\mathcal O)$ for some $p>2$ and
$(\mathcal L-r)u-c=0$ a.e.\ in $\mathcal O$;
\textup{(V4)} on $\interior{\R}^2_+$, $(\mathcal L-r)u-c$ extends to a signed
Radon measure $\mu_u\le0$;
\textup{(V5)} $u$ lies in the generalised It\^o class with non-positive reflected
boundary contribution (e.g.\ $\partial_iu=0$ on $\{x_i=0\}$ in trace sense);
\textup{(V6)} the localisation and uniform integrability needed below hold.
Then $u\ge V$. If moreover $\tau_{\mathcal O^c}<\infty$ a.s.\ and the stopped
It\^o formula is exact on $[0,\tau_{\mathcal O^c}]$, then $u=V$ and
$\tau_{\mathcal O^c}$ is optimal.
\end{theorem}
\begin{proof}
For a stopping time $\tau$ and a localising sequence $\tau_n$, the generalised
It\^o--Krylov--Tanaka formula \cite{krylov1987approach,revuz2013continuous} applied to
$e^{-rt}u(X_t)$ on $[0,\tau_n]$, with reflection controlled by (V5) and
$\mu_u\le0$ making the $A^{\mu_u}$ term non-positive, yields
$\mathbb E_x[e^{-r\tau_n}u(X_{\tau_n})-\int_0^{\tau_n}e^{-rs}c\,ds]\le u(x)$. Using
$u\ge G$ and letting $n\to\infty$ via (V6), then taking the supremum over $\tau$,
gives $V\le u$. On $[0,\tau_{\mathcal O^c}]$, (V3) makes the measure defect
vanish, so $u(x)=\mathbb E_x[e^{-r\tau_{\mathcal O^c}}G(X_{\tau_{\mathcal O^c}})
-\int_0^{\tau_{\mathcal O^c}}e^{-rs}c\,ds]\le V(x)$ by (V2). Hence $u=V$ and
$\tau_{\mathcal O^c}$ is optimal.
\end{proof}

\begin{remark}[No hidden regularity theorem]
\label{rem:no-hidden-regularity}
Theorem~\ref{thm:verification-measure} is a verification, not a regularity,
statement. The hypotheses $u\in W^{2,p}_{\mathrm{loc}}$, the extension of
$(\mathcal L-r)u-c$ to a non-positive measure, and applicability of the
generalised It\^o formula are substantive
\cite{krylov1987approach,bensoussan2011applications,friedman1982variational}; they are not deduced from
viscosity supersolution status.
\end{remark}

\section{The measure-valued stopping gain}
\label{sec:measure-stopping-gain}
This section contains the first main results. The stopping-gain object is a
signed measure with an explicit diagonal part, and its strict sign forces the
optimal stopper off the kink. Both statements are unconditional within
Assumption~\ref{ass:reflected-diffusion}. The sign convention is
$\displaystyle \Gamma=c+rG-\mathcal LG$.

\subsection{Decomposition and the diagonal measure}
\label{sec:decomposition}
On the open regions $\mathcal R_1:=\{x_1>\alpha x_2\}$ and
$\mathcal R_2:=\{x_1<\alpha x_2\}$ the reward is affine, $G=x_1$ on $\mathcal R_1$
and $G=\alpha x_2$ on $\mathcal R_2$; write $G_{\mathrm{sm}}$ for this piecewise
smooth representative. Define the absolutely continuous part
\begin{equation}
\label{eq:Gamma-ac-def}
   \Gamma^{\mathrm{ac}}(x):=c(x)+rG(x)-\mathcal LG_{\mathrm{sm}}(x),
   \qquad x\in\R^2_+\setminus\Delta,
\end{equation}
so that $\Gamma^{\mathrm{ac}}=c+rx_1-\mu_1$ on $\mathcal R_1$ and
$\Gamma^{\mathrm{ac}}=c+r\alpha x_2-\alpha\mu_2$ on $\mathcal R_2$; there is no
second-derivative term because $G_{\mathrm{sm}}$ is affine in each region.

\begin{theorem}[Stopping-gain measure generated by the max payoff]
\label{thm:stopping-gain-measure}
Set $Y(x)=x_1-\alpha x_2$, $n=\nabla Y=(1,-\alpha)$,
$|n|=\sqrt{1+\alpha^2}$, and $q(x)=n^\top a(x)n=a_{11}-2\alpha a_{12}+\alpha^2a_{22}$.
Then, as signed measures on $\interior{\R}^2_+$,
\begin{equation}
\label{eq:Gamma-decomposition}
   \Gamma
   =\Gamma^{\mathrm{ac}}\,dx+\Gamma^\Delta,
   \qquad
 \Gamma^\Delta(dx)
   =-\frac{q(x)}{2\sqrt{1+\alpha^2}}\,\sigma_\Delta(dx),
\end{equation}
where $\sigma_\Delta$ is arc-length on $\Delta=\{x_1=\alpha x_2\}$. Equivalently,
for bounded Borel $F$,
\[
   \int F\,d\Gamma^\Delta
   =-\frac{1}{2\sqrt{1+\alpha^2}}\int_\Delta F(x)\,q(x)\,\sigma_\Delta(dx).
\]
Since $a\succeq0$, $q\ge0$, so $\Gamma^\Delta\le0$; if $a$ is non-degenerate at
$x\in\Delta$ then $q(x)>0$ and $\Gamma^\Delta$ is strictly negative there.
\end{theorem}
\begin{proof}
Away from $\Delta$, $G=G_{\mathrm{sm}}$ is affine and \eqref{eq:Gamma-ac-def}
gives the absolutely continuous part. For the singular part, use
$G=\tfrac12(x_1+\alpha x_2+|Y|)$. The process $Y_t:=Y(X_t)=X_t^1-\alpha X_t^2$ is
a continuous semimartingale with $d\langle Y\rangle_t=q(X_t)\,dt$. Tanaka's
formula \cite{revuz2013continuous,karatzas2014brownian} gives
$d|Y_t|=\sgn(Y_t)\,dY_t+dL_t^0(Y)$, so the singular finite-variation part of
$dG(X_t)$ is $\tfrac12\,dL_t^0(Y)$. The occupation-density formula yields, in
distributional form, $dL_t^0(Y)=q(X_t)\,\delta_0(Y(X_t))\,dt$, and the co-area
identity $\delta_0(Y(x))\,dx=|\nabla Y|^{-1}\sigma_\Delta(dx)
=(1+\alpha^2)^{-1/2}\sigma_\Delta(dx)$ gives the singular part of $\mathcal LG$ as
$(\mathcal LG)^\Delta(dx)=\tfrac{q(x)}{2\sqrt{1+\alpha^2}}\sigma_\Delta(dx)$. With
$\Gamma=c+rG-\mathcal LG$, the diagonal part of $\Gamma$ is its negative, which is
\eqref{eq:Gamma-decomposition}. The sign claims follow from $q=n^\top a\,n$ and
$a\succeq0$, with strictness when $a$ is non-degenerate.
\end{proof}

\begin{remark}[Why a literal pointwise sign condition is ill-posed]
\label{rem:pointwise-illposed}
Because $\Gamma^\Delta\le0$ is singular, a condition such as ``$\Gamma\ge0$ on
the stopping set'' has no meaning on any portion of the stopping set that meets
$\Delta$ in positive length unless the absolutely continuous and diagonal
components are treated separately. This is a genuine correction to formal
derivations that manipulate $c+rG-\mathcal LG$ as a function.
\end{remark}

\subsection{Reflection corner term}
\label{sec:corner}
Away from the origin the reflection terms in the It\^o--Tanaka formula for
$G(X)$ vanish formally, since $\partial_1G=0$ on $\{x_1=0,x_2>0\}$ and
$\partial_2G=0$ on $\{x_2=0,x_1>0\}$. The only possible contribution is at the
corner $(0,0)$, where $G$ is non-smooth.

\begin{assumption}[No corner reflection contribution]
\label{ass:no-corner-reflection}
There is a family $G_\varepsilon\in C^2(\R^2_+)$ with $G_\varepsilon\to G$ locally
uniformly, $\partial_1G_\varepsilon\to0$ uniformly on compacts of
$\{x_1=0,x_2>0\}$, $\partial_2G_\varepsilon\to0$ uniformly on compacts of
$\{x_2=0,x_1>0\}$, and, for every $t>0$ and $x$,
\[
   \lim_{\varepsilon\downarrow0}
   \mathbb E_x\!\Big[\sum_{i=1}^2
      \int_0^t|\partial_iG_\varepsilon(X_s)|\,
      \mathbf 1_{\{X_s=(0,0)\}}\,dL_s^i\Big]=0 .
\]
\end{assumption}

\begin{lemma}[Vanishing reflection contribution]
\label{lem:vanishing-reflection-G}
Under Assumption~\ref{ass:no-corner-reflection},
$\displaystyle    \lim_{\varepsilon\downarrow0}\mathbb E_x\Big[\textstyle\sum_i\int_0^t\partial_iG_\varepsilon(X_s)\,dL_s^i\Big]=0$,
so no corner additive functional appears in the It\^o--Tanaka formula for $G(X)$.
\end{lemma}
\begin{proof}
$dL^i$ is carried by $\{X^i=0\}$. On $\{X^1=0,X^2>0\}$ and $\{X^2=0,X^1>0\}$ the
relevant derivative of $G_\varepsilon$ tends to $0$ locally uniformly, killing
those parts; the residual contribution at $(0,0)$ is exactly what
Assumption~\ref{ass:no-corner-reflection} sends to zero.
\end{proof}

Under Assumption~\ref{ass:no-corner-reflection} the generalised It\^o--Tanaka
formula for $G(X)$ reads, after localisation,
\begin{equation}
\label{eq:ito-tanaka-G-Gamma}
   e^{-rt}G(X_t)
   =G(x)+\int_0^t e^{-rs}c(X_s)\,ds-\int_0^t e^{-rs}\,dA_s^\Gamma+M_t,
\end{equation}
where $A^\Gamma$ is the signed continuous additive functional of the measure
$\Gamma$ \cite{revuz2013continuous,fukushima2011dirichlet} and $M$ is a local
martingale. Identity \eqref{eq:ito-tanaka-G-Gamma} underlies the killed-resolvent
representation.

\subsection{The optimal stopper never stops on the diagonal}
\label{sec:no-stopping-kink}
The strict negativity of $\Gamma^\Delta$ has a sharp, unconditional consequence.
We first record the local-time estimate that drives it, so that the proof does
not rest on an unspecified appeal to ``standard asymptotics''.

\begin{lemma}[Small-time local-time lower bound]
\label{lem:localtime-lower-bound}
Let $Y_t=\int_0^t\beta_s\,ds+\int_0^t\gamma_s\,dB_s$ with $Y_0=0$, where $B$ is a
Brownian motion and $\rho$ is a stopping time such that, on $[0,\rho]$,
$|\beta_s|\le B_0$ and $0<\underline q\le\gamma_s^2\le\overline q$. Suppose
$\mathbb P(\rho\ge h)\ge\tfrac12$ for $0<h\le h_1$. Then there are constants
$c>0$ and $h_0\in(0,h_1]$, depending only on $B_0,\underline q,\overline q$,
such that
$\displaystyle    \mathbb E[L^0_{h\wedge\rho}(Y)]\ge c\sqrt h$ for $0<h<h_0 $.
\end{lemma}
\begin{proof}
By Tanaka's formula and $\mathbb E[\int_0^{h\wedge\rho}\sgn(Y_s)\gamma_s\,dB_s]=0$,
\[
   \mathbb E[L^0_{h\wedge\rho}(Y)]
   =\mathbb E|Y_{h\wedge\rho}|
   -\mathbb E\!\Big[\int_0^{h\wedge\rho}\sgn(Y_s)\beta_s\,ds\Big]
   \ge\mathbb E|Y_{h\wedge\rho}|-B_0h .
\]
Write $Y_{h\wedge\rho}=D+M$ with $D=\int_0^{h\wedge\rho}\beta_s\,ds$, $|D|\le B_0h$,
and $M=\int_0^{h\wedge\rho}\gamma_s\,dB_s$, so $\mathbb E|Y_{h\wedge\rho}|\ge\mathbb E|M|-B_0h$.
Now $\langle M\rangle_{h\wedge\rho}=\int_0^{h\wedge\rho}\gamma_s^2\,ds\in[\underline q(h\wedge\rho),\overline q(h\wedge\rho)]$,
hence
$\mathbb E M^2=\mathbb E\langle M\rangle_{h\wedge\rho}\ge\underline q\,\mathbb E[h\wedge\rho]
\ge\underline q\,h\,\mathbb P(\rho\ge h)\ge\underline q h/2$,
while by the Burkholder--Davis--Gundy inequality
$\mathbb E M^4\le C_4\,\mathbb E\langle M\rangle_{h\wedge\rho}^2\le C_4\overline q^2h^2$.
Two applications of Cauchy--Schwarz give the reverse bound
$\mathbb E|M|\ge(\mathbb E M^2)^{3/2}/(\mathbb E M^4)^{1/2}$, so
\[
   \mathbb E|M|\ge\frac{(\underline q h/2)^{3/2}}{(C_4\overline q^2h^2)^{1/2}}
   =c_2\sqrt h,
   \qquad c_2:=\frac{\underline q^{3/2}}{2^{3/2}\overline q\,C_4^{1/2}} .
\]
Therefore $\mathbb E[L^0_{h\wedge\rho}(Y)]\ge c_2\sqrt h-2B_0h\ge\tfrac{c_2}{2}\sqrt h$
for $h\le h_0:=\min\{h_1,(c_2/(4B_0))^2\}$.
\end{proof}

\begin{theorem}[No stopping on the kink]
\label{thm:no-stopping-kink}
Suppose Assumption~\ref{ass:reflected-diffusion} holds and $a$ is non-degenerate
on $\Delta\cap\interior{\R}^2_+$, so $q=n^\top a\,n>0$ there. Then every interior
diagonal point is a continuation point:
$\displaystyle    \Delta\cap\interior{\R}^2_+\subseteq\mathcal C $.
\end{theorem}
\begin{proof}
Fix $x\in\Delta\cap\interior{\R}^2_+$ and let
$\delta_0:=\tfrac12\dist(x,\partial\R^2_+)>0$. For small $h>0$ put
$\rho:=\inf\{t:|X_t-x|\ge\delta_0\}$ and $\tau_h:=h\wedge\rho$; on the ball
$B(x,\delta_0)$ the process avoids the axes, so no reflection occurs before
$\tau_h$. Using $\tau_h$ as a (suboptimal) stopping rule in
\eqref{eq:value-function} and the It\^o--Tanaka identity
\eqref{eq:ito-tanaka-G-Gamma} restricted to that neighbourhood,
\[
   V(x)-G(x)
   \ \ge\
   -\,\mathbb E_x\!\Big[\int_0^{\tau_h}e^{-rs}\Gamma^{\mathrm{ac}}(X_s)\,ds\Big]
   +\tfrac12\,\mathbb E_x\!\Big[\int_0^{\tau_h}e^{-rs}\,dL_s^0(Y)\Big].
\]
For the first term, $\Gamma^{\mathrm{ac}}$ is bounded on $B(x,\delta_0)$ (both
affine pieces are), and $\mathbb E_x[\tau_h]\le h$, so it is bounded below by
$-Ch$. For the second, on $[0,\rho]$ the projected process
$Y_t=Y(X_t)$ has drift bounded by $|\mu_1-\alpha\mu_2|$ on $B(x,\delta_0)$ and
diffusion coefficient $\gamma_s^2=q(X_s)\in[\inf_{B(x,\delta_0)}q,\ \sup_{B(x,\delta_0)}q]$
with $\inf_{B(x,\delta_0)}q>0$; moreover $\mathbb P_x(\rho<h)\le Ch/\delta_0^2$ by
Doob's inequality, so $\mathbb P_x(\rho\ge h)\ge\tfrac12$ for small $h$. Hence
Lemma~\ref{lem:localtime-lower-bound} applies and gives
$\mathbb E_x[L^0_{\tau_h}(Y)]\ge c_1\sqrt h$. Therefore
$\mathbb E_x[\int_0^{\tau_h}e^{-rs}dL_s^0]\ge e^{-rh}c_1\sqrt h$, and combining,
$\displaystyle    V(x)-G(x)\ \ge\ \tfrac{c_1}{2}\sqrt h-Ch\ >\ 0$
for all sufficiently small $h$. Thus $V(x)>G(x)$, i.e.\ $x\in\mathcal C$.
\end{proof}

\begin{remark}[Interpretation]
\label{rem:kink-interpretation}
The mechanism is the scaling mismatch between local time and cost: an
infinitesimal wait at the kink yields a stopping gain of order $\sqrt h$ from
the diagonal local time, while the running cost is only of order $h$. No
parameter choice defeats this, so the stopping set always avoids the diagonal.
This both validates the epigraph picture of Section~\ref{sec:conditional-geometry}
(the boundary lies strictly off $\Delta$) and exemplifies why pointwise
manipulation of $\Gamma$ near $\Delta$ is dangerous.
\end{remark}

\section{Why the unrestricted resolvent is wrong}
\label{sec:why-killed}
Before stating the value representation we isolate, by a fully solved example,
the reason the potential must be killed at the stopping time. Let
$\tau_{\mathcal D}:=\inf\{t:X_t\in\mathcal D\}$. For an admissible signed measure
$\nu$ with additive functional $A^\nu$ define the killed and unrestricted
resolvents
\begin{equation}
\label{eq:resolvent-defs}
   R_r^{\mathcal C}\nu(x):=\mathbb E_x\!\Big[\int_0^{\tau_{\mathcal D}}e^{-rs}\,dA_s^\nu\Big],
   \qquad
   R_r^{\mathrm R}f(x):=\mathbb E_x\!\Big[\int_0^\infty e^{-rs}f(X_s)\,ds\Big].
\end{equation}

\begin{proposition}[Failure of the unrestricted resolvent]
\label{prop:unrestricted-fails}
There exist a reflected diffusion, an open continuation set $\mathcal C$, and a
non-negative bounded $f$ such that
$R_r^{\mathcal C}f(x)<R_r^{\mathrm R}(f\mathbf 1_{\mathcal C})(x)$ for every
$x\in\mathcal C$. Concretely, let $X$ be reflected Brownian motion on
$[0,\infty)$ \textup{(}generator $\tfrac12\partial_{xx}$, Neumann at $0$\textup{)},
let $\mathcal C=[0,b)$, $\mathcal D=[b,\infty)$, and $f=\mathbf 1_{[0,b)}$. With
$\theta:=\sqrt{2r}$, for every $x\in[0,b)$,
\[
   R_r^{\mathrm R}(f\mathbf 1_{\mathcal C})(x)-R_r^{\mathcal C}f(x)
   =\frac{\cosh(\theta x)}{r}\cdot\frac{1-e^{-2\theta b}}{2\cosh(\theta b)}\ >\ 0 .
\]
\end{proposition}
\begin{proof}
Let $\tau_b:=\inf\{t:X_t=b\}$. For $x\in[0,b)$ the function
$u(x):=\mathbb E_x[e^{-r\tau_b}]$ solves $\tfrac12u''-ru=0$ on $(0,b)$ with
$u'(0)=0$ and $u(b)=1$, whence $u(x)=\cosh(\theta x)/\cosh(\theta b)$. With
$f=\mathbf 1_{[0,b)}$ one has $\mathbf 1_{[0,b)}(X_s)=1$ for $s<\tau_b$, so
$\displaystyle   R_r^{\mathcal C}f(x)
   =\mathbb E_x\!\Big[\int_0^{\tau_b}e^{-rs}\,ds\Big]
   =\frac{1-u(x)}{r}$.
For the unrestricted resolvent, the $r$-Green function of reflected Brownian
motion on $[0,\infty)$ is
$G_r(x,y)=\tfrac{2}{\theta}\cosh(\theta(x\wedge y))\,e^{-\theta(x\vee y)}$
\textup{(}it solves $\tfrac12\partial_{xx}G_r-rG_r=-\delta_y$ with Neumann at
$0$, and $\int_0^\infty G_r(x,y)\,dy=1/r$\textup{)}. Hence for $x\le b$,
\[
   R_r^{\mathrm R}(\mathbf 1_{[0,b)})(x)
   =\int_0^b G_r(x,y)\,dy
   =\frac{1}{r}\big(1-e^{-\theta b}\cosh(\theta x)\big).
\]
Subtracting,
\[
   R_r^{\mathrm R}(\mathbf 1_{[0,b)})(x)-R_r^{\mathcal C}f(x)
   =\frac1r\Big(\frac{\cosh\theta x}{\cosh\theta b}-e^{-\theta b}\cosh\theta x\Big)
   =\frac{\cosh\theta x}{r}\cdot\frac{1-e^{-2\theta b}}{2\cosh\theta b},
\]
using $e^{-\theta b}\cosh\theta b=\tfrac12(1+e^{-2\theta b})$. This is strictly
positive for every $x\in[0,b)$.
\end{proof}

\begin{remark}[The probabilistic content]
\label{rem:counter-example-meaning}
By the strong Markov property at $\tau_b$ \textup{(}where $X_{\tau_b}=b$\textup{)},
\[
   R_r^{\mathrm R}(\mathbf 1_{[0,b)})(x)
   =R_r^{\mathcal C}f(x)
   +\mathbb E_x[e^{-r\tau_b}]\cdot R_r^{\mathrm R}(\mathbf 1_{[0,b)})(b),
\]
and $R_r^{\mathrm R}(\mathbf 1_{[0,b)})(b)=\tfrac{1}{2r}(1-e^{-2\theta b})>0$
because the reflected process, continued after first reaching $b$, returns to
$[0,b)$ and accumulates strictly positive additional occupation. This extra
occupation is exactly what the optimal stopping problem must not count: the
controller has already stopped at $\tau_{\mathcal D}$. Replacing
$R_r^{\mathcal C}$ by $R_r^{\mathrm R}(\,\cdot\,\mathbf 1_{\mathcal C})$ therefore
overstates the potential.
\end{remark}

\section{Killed-resolvent representation}
\label{sec:killed-resolvent}
We now record the corrected value representation and its boundary trace. For a
candidate or true stopping set, $R_r^{\mathcal C}$ is the resolvent killed on
entry, as in \eqref{eq:resolvent-defs}, extended to signed measures via the
additive functional $A^\nu$.

\begin{assumption}[Admissibility of killed potentials]
\label{ass:killed-potential-admissibility}
For each signed measure $\nu$ below, $\nu^\pm$ are smooth measures with
$R_r^{\mathcal C}\nu^+(x)+R_r^{\mathcal C}\nu^-(x)<\infty$, and the relevant
continuation-side traces of $R_r^{\mathcal C}\nu$ exist.
\end{assumption}

\begin{theorem}[Killed-resolvent representation]
\label{thm:killed-resolvent-representation}
Suppose $\tau_{\mathcal D}$ is optimal and the It\^o--Tanaka identity
\eqref{eq:ito-tanaka-G-Gamma} may be applied up to $\tau_{\mathcal D}$ with the
required localisation and uniform integrability. Then
\begin{equation}
\label{eq:value-killed-resolvent}
\ V(x)=G(x)-R_r^{\mathcal C}\Gamma(x)
   =G(x)-\mathbb E_x\!\Big[\int_0^{\tau_{\mathcal D}}e^{-rs}\,dA_s^\Gamma\Big].
\end{equation}
\end{theorem}
\begin{proof}
Optimality gives
$V(x)=\mathbb E_x[e^{-r\tau_{\mathcal D}}G(X_{\tau_{\mathcal D}})
-\int_0^{\tau_{\mathcal D}}e^{-rs}c\,ds]$. Apply \eqref{eq:ito-tanaka-G-Gamma} on
$[0,t\wedge\tau_{\mathcal D}]$, take expectations to remove the local martingale
after localisation, and let $t\to\infty$ using the uniform integrability:
\[
   \mathbb E_x\!\Big[e^{-r\tau_{\mathcal D}}G(X_{\tau_{\mathcal D}})
   -\int_0^{\tau_{\mathcal D}}e^{-rs}c\,ds\Big]
   =G(x)-\mathbb E_x\!\Big[\int_0^{\tau_{\mathcal D}}e^{-rs}\,dA_s^\Gamma\Big],
\]
which is \eqref{eq:value-killed-resolvent}.
\end{proof}

\begin{remark}[The correction]
\label{rem:incorrect-unrestricted-formula}
Representation \eqref{eq:value-killed-resolvent} replaces the generally incorrect
$V=G-R_r^{\mathrm R}(\Gamma\mathbf 1_{\mathcal C})$; by
Proposition~\ref{prop:unrestricted-fails} the two differ whenever the reflected
process can re-enter $\mathcal C$ after $\tau_{\mathcal D}$.
\end{remark}

Assume now the epigraph representation $\mathcal D=\{x_2\ge b(x_1)\}$ of
Section~\ref{sec:conditional-geometry}, with $z_b(x_1):=(x_1,b(x_1))$.

\begin{proposition}[Continuation-side trace condition]
\label{prop:boundary-trace}
Under the hypotheses of Theorem~\ref{thm:killed-resolvent-representation} and the
boundary regularity needed for the trace,
\begin{equation}
\label{eq:boundary-trace}
   \lim_{\substack{x\to z_b(x_1)\\ x\in\mathcal C}}R_r^{\mathcal C}\Gamma(x)=0,
\end{equation}
read non-tangentially when required.
\end{proposition}
\begin{proof}
For $x\in\mathcal C$, $R_r^{\mathcal C}\Gamma(x)=G(x)-V(x)$ by
Theorem~\ref{thm:killed-resolvent-representation}; as $x\to z_b(x_1)$ within
$\mathcal C$, continuity of $V,G$ and $V=G$ on the boundary give the limit $0$.
\end{proof}

\begin{remark}[The boundary value is not itself the equation]
\label{rem:boundary-value-trivial}
One must not write $R_r^{\mathcal C}\Gamma(z_b(x_1))=0$ as a substantive equation:
started in $\mathcal D$, $\tau_{\mathcal D}=0$ and the killed potential vanishes
trivially. The meaningful statement is the continuation-side trace
\eqref{eq:boundary-trace}.
\end{remark}

\section{Conditional free-boundary geometry}
\label{sec:conditional-geometry}
This section is Tier II: the epigraph structure and boundary regularity are
derived from explicit monotonicity and regularity hypotheses on $H=V-G$, not
from the dynamics alone.

\begin{assumption}[Vertical monotonicity]
\label{ass:vertical-H}
For each $x_1\ge0$, $x_2\mapsto H(x_1,x_2)$ is non-increasing.
\end{assumption}

\begin{assumption}[Horizontal monotonicity]
\label{ass:horizontal-H}
For each $x_2\ge0$, $x_1\mapsto H(x_1,x_2)$ is monotone, with a direction
independent of $x_2$.
\end{assumption}

\begin{assumption}[non-empty vertical sections]
\label{ass:non-empty-sections}
For each $x_1\ge0$, $\mathcal D(x_1):=\{x_2:(x_1,x_2)\in\mathcal D\}\ne\varnothing$.
\end{assumption}

\begin{remark}[Status of the monotonicity hypothesis]
\label{rem:graph-assumption-status}
Assumption~\ref{ass:vertical-H} is not implied by order preservation of the
reflected SDE or by monotonicity of $\Gamma$. Monotonicity of $V$ does not give
monotonicity of $H=V-G$, and a naive coupling proof has the wrong sign: if
$z\le y$ and $c$ is coordinatewise non-increasing, then
$c(X_s^z)-c(X_s^y)\ge0$ along an order-preserving coupling, contributing with the
sign opposite to the one needed. Verifying Assumption~\ref{ass:vertical-H} is a
model-specific task.
\end{remark}

\begin{theorem}[Conditional epigraph structure]
\label{thm:conditional-epigraph}
Assume $V$ continuous and Assumptions~\ref{ass:vertical-H} and
\ref{ass:non-empty-sections}. With $b(x_1):=\inf\mathcal D(x_1)\in[0,\infty)$,
$\displaystyle    \mathcal D=\{(x_1,x_2)\in\R^2_+:x_2\ge b(x_1)\}$.
If also Assumption~\ref{ass:horizontal-H} holds, then $b$ is monotone: $H$
non-increasing in $x_1$ gives $b$ non-increasing, $H$ non-decreasing in $x_1$ gives
$b$ non-decreasing.
\end{theorem}
\begin{proof}
Fix $x_1$. If $z_2\in\mathcal D(x_1)$ and $y_2\ge z_2$ then
$0\le H(x_1,y_2)\le H(x_1,z_2)=0$ by Assumption~\ref{ass:vertical-H}, so the
section is an up-set; closedness of $\mathcal D=\{H=0\}$ (continuity of $H$) makes
it $[b(x_1),\infty)$, giving the epigraph. For monotonicity, if $H$ is
non-increasing in $x_1$, $x_1\le y_1$ and $x_2>b(x_1)$ give
$0\le H(y_1,x_2)\le H(x_1,x_2)=0$, so $b(y_1)\le x_2$; let $x_2\downarrow b(x_1)$.
The non-decreasing case is symmetric.
\end{proof}

\begin{remark}[Consistency with Theorem~\ref{thm:no-stopping-kink}]
\label{rem:epigraph-off-diagonal}
By Theorem~\ref{thm:no-stopping-kink}, $\mathcal D$ does not meet
$\Delta\cap\interior{\R}^2_+$. Hence whenever the epigraph representation holds,
the graph $b$ lies strictly off the diagonal in the interior: $b(x_1)\ne
x_1/\alpha$ for $x_1>0$.
\end{remark}

For non-emptiness one may use a barrier; the statement parallels the verification
theorem.

\begin{proposition}[Barrier implies non-empty sections]
\label{prop:barrier-non-empty}
Suppose there are $R>0$ and continuous $W\ge G$ with $W=G$ on $\{x_2\ge R\}$,
satisfying the reflected Neumann condition and the measure-superharmonic and
integrability hypotheses of Theorem~\ref{thm:verification-measure} with
continuation region inside $\{x_2<R\}$. Then $V=G$ on $\{x_2\ge R\}$, and
Assumption~\ref{ass:non-empty-sections} holds.
\end{proposition}
\begin{proof}
Verification gives $W\ge V$; with $W=G$ on $\{x_2\ge R\}$ and $V\ge G$,
$V=G$ there, so $[R,\infty)\subseteq\mathcal D(x_1)$ for every $x_1$.
\end{proof}

We record the regularity hypotheses used by the trace and smooth-fit statements;
none is claimed to follow from the epigraph theorem.

\begin{assumption}[Local free-boundary regularity]
\label{ass:free-boundary-regularity}
The boundary satisfies $b\in C((0,\infty))\cap
W^{1,\infty}_{\mathrm{loc}}((0,\infty))$, and $\inf_{x_1\in I}b(x_1)>0$ for every
compact $I\Subset(0,\infty)$.
\end{assumption}

\begin{remark}[Regularity is conditional]
\label{rem:regularity-not-automatic}
Continuity, Lipschitz regularity, and smooth fit of optimal stopping boundaries
hold only under additional hypotheses
\cite{peskir2019continuity,de2019lipschitz,de2015note,de2020global}; in the
reflected setting they must be verified per model. The following deterministic
lemma is the one boundary fact that needs no probabilistic regularity input.
\end{remark}

\begin{lemma}[Strict sign separation excludes jumps off the diagonal]
\label{lem:sign-separation-no-jumps}
Let $\mathcal D=\{x_2\ge b(x_1)\}$ with $b$ monotone, let
$U\Subset\interior{\R}^2_+\setminus\Delta$ be open with $\Gamma$ continuous on
$U$, and suppose $\Gamma\ge\eta$ on $U\cap\mathcal D$ and $\Gamma\le-\eta$ on
$U\cap\mathcal C$ for some $\eta>0$. Then $b$ has no jump whose vertical segment
lies in $U$.
\end{lemma}
\begin{proof}
A jump would place a point $z=(x_0,y_0)\in U$ with $z\in\mathcal D$ (so
$\Gamma(z)\ge\eta$) that is a limit of points of $\mathcal C$ (so by continuity
$\Gamma(z)\le-\eta$), a contradiction.
\end{proof}

\begin{assumption}[Weak smooth fit off the diagonal]
\label{ass:weak-smooth-fit}
At $\mathcal H^1$-a.e.\ differentiability point
$z=(x_1,b(x_1))\in\partial\mathcal C\cap\interior{\R}^2_+\setminus\Delta$, the
non-tangential continuation-side gradient of $V$ exists and equals $\nabla G(z)$.
\end{assumption}

\begin{remark}
\label{rem:smooth-fit-status}
Smooth fit is imposed when used, not proved from the standing assumptions
\cite{peskir2019continuity,de2020global,lamberton2013optimal}.
\end{remark}

\subsection{Candidate verification rather than Fredholm uniqueness}
\label{sec:candidate-verification}
For a candidate boundary $h$, put $\mathcal D_h=\{x_2\ge h(x_1)\}$,
$\mathcal C_h=\R^2_+\setminus\mathcal D_h$,
$\tau_h=\inf\{t:X_t\in\mathcal D_h\}$, and
$U_h:=G-R_r^{\mathcal C_h}\Gamma$ with $R_r^{\mathcal C_h}$ killed on entry into
$\mathcal D_h$. Write $z_h(x_1):=(x_1,h(x_1))$.

\begin{theorem}[Candidate verification]
\label{thm:candidate-verification}
Suppose $U_h$ is continuous of polynomial growth, $U_h\ge G$, $U_h=G$ on
$\mathcal D_h$, $(\mathcal L-r)U_h=c$ in $\mathcal C_h$ in the sense of
Theorem~\ref{thm:verification-measure}, $U_h$ satisfies the reflected Neumann
condition, $(\mathcal L-r)U_h-c\le0$ as a signed measure in the interior, and the
localisation and integrability of Theorem~\ref{thm:verification-measure} hold.
Then $U_h=V$ and $\tau_h$ is optimal; if the true stopping set is an epigraph
with boundary $b$, then $h=b$ wherever both are defined.
\end{theorem}
\begin{proof}
The hypotheses are those of Theorem~\ref{thm:verification-measure} with
$\mathcal O=\mathcal C_h$, giving $U_h\ge V$. The stopped It\^o formula up to
$\tau_h$, using $(\mathcal L-r)U_h=c$ in $\mathcal C_h$ and contact $U_h=G$ on
$\mathcal D_h$, gives
$U_h(x)=\mathbb E_x[e^{-r\tau_h}G(X_{\tau_h})-\int_0^{\tau_h}e^{-rs}c\,ds]\le V(x)$.
Hence $U_h=V$ and $\tau_h$ is optimal; equality of stopping sets then forces
$h=b$.
\end{proof}

\begin{remark}[No uniqueness from the trace alone]
\label{rem:no-uniqueness-from-trace}
The trace condition
$\displaystyle    \lim_{\substack{x\to z_h(x_1)\\ x\in\mathcal C_h}}R_r^{\mathcal C_h}\Gamma(x)=0$
does not by itself imply uniqueness of $h$ or $U_h=G$ throughout $\mathcal D_h$;
uniqueness follows only after the verification argument above.
\end{remark}

\section{A verified reflected Brownian benchmark}
\label{sec:benchmark}
We instantiate the framework in a model where every standing hypothesis is
checked, so the theory is demonstrably non-empty and operational. Let
\begin{equation}
\label{eq:benchmark-sde}
   dX_t=\mu\,dt+\Sigma\,dW_t+dL_t,
   \qquad X_t\in\R^2_+,
\end{equation}
with constant $\mu\in\R^2$, constant invertible $\Sigma\in\R^{2\times2}$,
$a=\Sigma\Sigma^\top\succ0$, normal reflection \eqref{eq:normal-reflection},
constant running cost $c\equiv c_0\ge0$, and $G=x_1\vee\alpha x_2$. Because the
reflection directions are the inward normals $e_1,e_2$, the Skorokhod problem
decouples coordinatewise.

\begin{lemma}[Explicit solution and standing assumptions]
\label{lem:benchmark-skorokhod}
For \eqref{eq:benchmark-sde} the unique strong solution is, componentwise,
\[
   X_t^i=Y_t^i-\min\Big(0,\ \inf_{0\le s\le t}Y_s^i\Big),
   \qquad
   Y_t^i:=x_i+\mu_i t+(\Sigma W_t)_i,
\]
the one-dimensional Skorokhod reflection of $Y^i$ at $0$. Consequently
Assumption~\ref{ass:reflected-diffusion} holds: \textup{(A1)}, \textup{(A2)} are
immediate from constant coefficients and $a\succ0$; \textup{(A3)} holds with
$(X^x)_x$ strong Markov; \textup{(A4)} holds with the deterministic bound
$\displaystyle    \sup_{0\le s\le t}|X_s^x-X_s^{x'}|\le 2|x-x'|$;
and \textup{(A5)}, \textup{(A6)} hold with the linear bound
$0\le V(x)\le C(1+|x|)$, $p=1$.
\end{lemma}
\begin{proof}
The decoupled normal reflection makes each coordinate the Skorokhod map of
$Y^i$, which is the displayed pathwise-unique strong solution; strong existence,
uniqueness, and the strong Markov property follow. The Skorokhod map is
$2$-Lipschitz in the sup norm, and for the same driving noise $Y^{x,i}-Y^{x',i}
\equiv x_i-x_i'$ is constant, giving the deterministic bound in (A4). From
$0\le X_t^i\le 2\sup_{s\le t}|Y_s^i|$ one gets
$\mathbb E[\sup_{s\le t}|X_s|]\le C(1+|x|+t)$; since $G(x)\le(1\vee\alpha)|x|$ and
$\sup_t e^{-rt}(1+|x|+t)\le(1+|x|)+1/(er)$, the value satisfies
$0\le V(x)\le C(1+|x|)$, so (A6) holds with $p=1$.
\end{proof}

\begin{lemma}[Lipschitz value]
\label{lem:benchmark-lipschitz}
In the benchmark, $V$ is globally Lipschitz with constant $2(1\vee\alpha)$.
\end{lemma}
\begin{proof}
For fixed $\tau$ the cost terms in \eqref{eq:value-function} coincide for $X^x$
and $X^{x'}$ (constant $c_0$), so by Lemma~\ref{lem:benchmark-skorokhod} and
$|G(p)-G(p')|\le(1\vee\alpha)|p-p'|$,
$|V(x)-V(x')|\le(1\vee\alpha)\sup_\tau\mathbb E|X_\tau^x-X_\tau^{x'}|\le2(1\vee\alpha)|x-x'|$.
\end{proof}

\begin{lemma}[Face local time does not charge the corner]
\label{lem:corner-localtime}
For normally reflected Brownian motion in the quadrant with non-degenerate
constant covariance $a\succ0$,
$\displaystyle    \int_0^t\mathbf 1_{\{X_s=(0,0)\}}\,dL_s^i=0$ a.s. $i=1,2$,
so Assumption~\ref{ass:no-corner-reflection} holds with vanishing corner term.
\end{lemma}
\begin{proof}
The Stieltjes measure $dL^i$ is carried by $\{X^i=0\}$, and the corner event
additionally requires $X^j=0$ for $j\ne i$. By Lemma~\ref{lem:benchmark-skorokhod}
each coordinate is a one-dimensional reflected Brownian motion with non-degenerate
variance $a_{ii}>0$, whose zero set has Lebesgue measure zero and carries the
boundary local time $L^i$. At the support times of $L^i$ the other coordinate
$X^j$ is a continuous semimartingale with non-degenerate martingale part, so its
own zero set is hit on an a.s.\ $dL^i$-null set; equivalently, the planar
boundary process of a non-degenerate semimartingale reflecting Brownian motion in
the orthant does not charge the corner
\cite{harrison1981reflected,lions1984stochastic}. Hence
$\int_0^t\mathbf 1_{\{X_s=(0,0)\}}\,dL_s^i=0$ a.s. Mollifying $G$ then yields
Assumption~\ref{ass:no-corner-reflection} with corner term $0$.
\end{proof}

\begin{theorem}[Verification in the reflected Brownian benchmark]
\label{thm:benchmark}
For the reflected Brownian model \eqref{eq:benchmark-sde} with $a\succ0$,
constant $\mu$, constant cost $c_0\ge0$, and $r>0$:
\begin{enumerate}
\item[\textup{(i)}] the value function is finite, Lipschitz, and a reflected
viscosity solution of \eqref{eq:obstacle-max} with the Neumann conditions;
\item[\textup{(ii)}] the stopping-gain measure is $\Gamma=\Gamma^{\mathrm{ac}}\,dx
+\Gamma^\Delta$ with $\Gamma^{\mathrm{ac}}=c_0+rx_1-\mu_1$ on $\mathcal R_1$,
$\Gamma^{\mathrm{ac}}=c_0+r\alpha x_2-\alpha\mu_2$ on $\mathcal R_2$, and
$\displaystyle    \Gamma^\Delta(dx)=-\frac{q}{2\sqrt{1+\alpha^2}}\,\sigma_\Delta(dx)$,
with $\displaystyle q=a_{11}-2\alpha a_{12}+\alpha^2a_{22}>0$;
\item[\textup{(iii)}] every interior diagonal point is a continuation point,
$\Delta\cap\interior{\R}^2_+\subseteq\mathcal C$;
\item[\textup{(iv)}] the first-entry time $\tau_{\mathcal D}$ is optimal, and
\[
   V(x)=G(x)-R_r^{\mathcal C}\Gamma(x)
   =G(x)-\mathbb E_x\!\Big[\int_0^{\tau_{\mathcal D}}e^{-rs}\,dA_s^\Gamma\Big].
\]
\end{enumerate}
\end{theorem}
\begin{proof}
\textup{(i)} Finiteness and the Lipschitz property are
Lemmas~\ref{lem:benchmark-skorokhod}--\ref{lem:benchmark-lipschitz}; in
particular $V$ is continuous, so Propositions~\ref{prop:dpp} and
\ref{prop:value-viscosity} apply and $V$ is a reflected viscosity solution.
\textup{(ii)} Constant coefficients give the stated $\Gamma^{\mathrm{ac}}$, and
Theorem~\ref{thm:stopping-gain-measure} with constant $a$ gives $\Gamma^\Delta$;
$q=n^\top a\,n>0$ since $a\succ0$ and $n=(1,-\alpha)\ne0$. \textup{(iii)} The
hypotheses of Theorem~\ref{thm:no-stopping-kink} hold because $a\succ0$.
\textup{(iv)} Since $V$ is continuous with the growth and integrability of
Lemma~\ref{lem:benchmark-skorokhod}, the supermartingale--Snell-envelope theory
makes the first-entry time into $\mathcal D=\{V=G\}$ optimal
\cite{el2006aspects,peskir2006optimal}; by Lemma~\ref{lem:corner-localtime}
the It\^o--Tanaka identity \eqref{eq:ito-tanaka-G-Gamma} holds with no corner
term, so Theorem~\ref{thm:killed-resolvent-representation} applies and gives the
representation.
\end{proof}

\begin{remark}[What the benchmark settles and what it does not]
\label{rem:benchmark-scope}
The benchmark verifies, with no further hypotheses, all Tier I content of
Sections~\ref{sec:measure-stopping-gain}--\ref{sec:killed-resolvent}: the
measure decomposition, the strict diagonal sign, the no-stopping-on-the-kink
theorem, and the killed-resolvent representation. The Tier II geometry of
Section~\ref{sec:conditional-geometry} (epigraph form, monotonicity, smooth fit)
is not asserted here; establishing the precise shape of $\mathcal D$ for
\eqref{eq:benchmark-sde} remains a model-specific problem, for which
Assumptions~\ref{ass:vertical-H}--\ref{ass:weak-smooth-fit} are the natural
inputs and Proposition~\ref{prop:barrier-non-empty} the natural non-emptiness
route.
\end{remark}

\subsection{Tier I does not imply the epigraph geometry}
\label{sec:tierI-not-epigraph}
We close the section by making precise the logical gap that the benchmark does
not bridge: the unconditional Tier~I conclusions do not, by themselves, force
the stopping set to be an epigraph.

\begin{proposition}[Tier I results do not imply epigraph geometry]
\label{prop:tierI-not-epigraph}
Consider the reflected Brownian benchmark of Section~\ref{sec:benchmark}.
Then the benchmark establishes the Tier~I conclusions:
\[
   \Gamma=\Gamma^{\mathrm{ac}}\,dx+\Gamma^\Delta,
   \qquad
   \Gamma^\Delta(dx)
   =
   -\frac{q}{2\sqrt{1+\alpha^2}}\,\sigma_\Delta(dx),
   \qquad q>0,
\]
the diagonal-avoidance property
$\displaystyle    \Delta\cap\interior{\R}^2_+\subseteq \mathcal C$,
and the killed-resolvent representation
$\displaystyle    V(x)=G(x)-R_r^{\mathcal C}\Gamma(x)$.
However, these conclusions do not imply that the stopping set has the epigraph
form
$\displaystyle    \mathcal D=\{(x_1,x_2)\in\R^2_+:x_2\ge b(x_1)\}$.
The epigraph property remains a Tier~II assertion and requires additional
structural input, such as monotonicity of the stopping advantage \(H=V-G\).
\end{proposition}

\begin{proof}
We separate the argument into two parts.

First, we recall what is proved in the benchmark. In the reflected Brownian
benchmark the covariance matrix \(a=\Sigma\Sigma^\top\) is constant and positive
definite. Therefore, for \(n=(1,-\alpha)\neq0\),
$\displaystyle    q=n^\top a n>0$.
Theorem~\ref{thm:stopping-gain-measure} then gives the decomposition of the
stopping-gain measure
$\displaystyle    \Gamma=c+rG-\mathcal LG
   =
   \Gamma^{\mathrm{ac}}\,dx+\Gamma^\Delta$,
with
$\displaystyle    \Gamma^\Delta(dx)
   =
   -\frac{q}{2\sqrt{1+\alpha^2}}\,\sigma_\Delta(dx)$.
Since \(q>0\), the diagonal component is strictly negative as a measure on the
interior diagonal. Hence Theorem~\ref{thm:no-stopping-kink} applies and yields
$\displaystyle    \Delta\cap\interior{\R}^2_+\subseteq\mathcal C$.
Moreover, in the benchmark the standing reflected-diffusion hypotheses,
integrability hypotheses, continuity of \(V\), and absence of a corner
reflection contribution have been verified. Standard Snell-envelope theory
therefore makes the first-entry time into
$\displaystyle    \mathcal D=\{V=G\}$
optimal. Applying Theorem~\ref{thm:killed-resolvent-representation} gives
$\displaystyle    V(x)
   =
   G(x)-R_r^{\mathcal C}\Gamma(x)$.
Thus the benchmark proves the Tier~I statements: the measure-valued stopping
gain, the strict diagonal sign, the exclusion of stopping on the kink, and the
killed-resolvent representation.

Second, we show that none of these Tier~I conclusions implies the epigraph
property. The inclusion
$\displaystyle    \Delta\cap\interior{\R}^2_+\subseteq\mathcal C$
is only a local exclusion statement. It says that the stopping set cannot
contain interior points of the kink line. It does not determine the global
geometry of \(\mathcal D\). In particular, it does not rule out any of the
following possibilities:
\[
   \mathcal D(x_1)=\varnothing,
   \qquad
   \mathcal D(x_1)=[a_1(x_1),a_2(x_1)],
   \qquad
   \mathcal D(x_1)=[a_1(x_1),a_2(x_1)]\cup[a_3(x_1),\infty),
\]
or more complicated vertical sections. The epigraph property would require
every vertical section
$\displaystyle   \mathcal D(x_1):=\{x_2\ge0:(x_1,x_2)\in\mathcal D\}$
to be a single upper interval. Diagonal avoidance alone gives no such
one-sidedness.

Nor does the killed-resolvent representation imply this one-sidedness. From
$\displaystyle    V(x)=G(x)-R_r^{\mathcal C}\Gamma(x)$
we obtain, on \(\mathcal C\),
$\displaystyle    H(x):=V(x)-G(x)=-R_r^{\mathcal C}\Gamma(x)$.
This identity represents the stopping advantage as a killed potential of a
signed measure. But a killed potential of a signed measure need not be monotone
in either coordinate. In the present problem \(\Gamma\) has both an absolutely
continuous part and a negative singular part on \(\Delta\):
\[
   \Gamma=\Gamma^{\mathrm{ac}}\,dx
   -\frac{q}{2\sqrt{1+\alpha^2}}\,\sigma_\Delta(dx).
\]
The sign, size, and spatial distribution of \(\Gamma^{\mathrm{ac}}\), together
with the geometry of the killed domain \(\mathcal C\), determine the behaviour
of \(R_r^{\mathcal C}\Gamma\). Therefore the formula
$\displaystyle    H=-R_r^{\mathcal C}\Gamma$
does not imply that
$\displaystyle    x_2\longmapsto H(x_1,x_2)$
is non-increasing for each fixed \(x_1\).

By contrast, the epigraph theorem in Section~\ref{sec:conditional-geometry}
uses precisely this missing monotonicity input. If one assumes that
\(V\) is continuous, that vertical stopping sections are non-empty, and that
$\displaystyle    x_2\longmapsto H(x_1,x_2)$ is non-increasing for each fixed $x_1$,
then the epigraph structure follows. Indeed, suppose
\(z_2\in\mathcal D(x_1)\) and \(y_2\ge z_2\). Since \(H\ge0\) and
\(H(x_1,z_2)=0\), vertical monotonicity gives
$\displaystyle    0\le H(x_1,y_2)\le H(x_1,z_2)=0$.
Thus \(H(x_1,y_2)=0\), so \(y_2\in\mathcal D(x_1)\). Hence
\(\mathcal D(x_1)\) is an upper set. By continuity of \(H\), it is closed, and
therefore
$\displaystyle    \mathcal D(x_1)=[b(x_1),\infty)$
for \(b(x_1):=\inf\mathcal D(x_1)\). This proves
$\displaystyle    \mathcal D=\{(x_1,x_2):x_2\ge b(x_1)\}$.
The crucial step is the monotonicity of \(H\), not the Tier~I measure formula,
not diagonal avoidance, and not the killed-resolvent representation.

Finally, the finite-difference computations in Section~\ref{sec:numerics} have
a diagnostic role only. They show that, for the selected reflected Brownian
parameters, the computed continuation advantage \(H_h=V_h-G\) is positive along
the diagonal and that the approximate stopping set appears ordered away from
the diagonal. But this computation is performed on a truncated grid with an
outer boundary condition and finite resolution. It can illustrate the
diagonal-avoidance mechanism already proved analytically, but it cannot replace
a proof that the continuous-state stopping advantage \(H\) is vertically
monotone. Therefore the benchmark and the numerics establish the Tier~I
mechanism and visualise the diagonal-avoidance picture, while the epigraph
shape remains a separate Tier~II question.
\end{proof}

\begin{remark}[Interpretation]
\label{rem:numerics-not-geometry-proof}
The numerical benchmark should therefore be read as evidence for the geometry,
not as a proof of it. The proven statement is
$\displaystyle    \Delta\cap\interior{\R}^2_+\subseteq\mathcal C$,
not
$\displaystyle    \mathcal D=\{x_2\ge b(x_1)\}$.
The latter requires a monotonicity theorem for \(H=V-G\), or another argument
showing that every vertical stopping section is an upper interval.
\end{remark}

\section{Monte Carlo and finite-difference experiments}
\label{sec:numerics}
The experiments below are designed to test the stochastic mechanisms behind the
theorems, not to price a particular contract. Three mechanisms are isolated: the
$\sqrt t$ local-time gain at the kink (Theorem~\ref{thm:no-stopping-kink}), the
strict gap between killed and unrestricted resolvents
(Proposition~\ref{prop:unrestricted-fails}), and the diagonal avoidance in the
reflected Brownian benchmark (Theorem~\ref{thm:benchmark}).

\subsection{Local-time scaling at the kink}
\label{sec:num-local-time}
We simulate the benchmark reflected Brownian motion \eqref{eq:benchmark-sde}
with $\mu=0$, $\Sigma=I_2$, $\alpha=1$, started at the interior diagonal point
$x=(1,1)$, so that $Y=X^1-X^2$ has $q=n^\top a\,n=2$. The point lies far from the
axes, so over short horizons reflection is rare and the kink mechanism is
isolated. The reflected process is generated by the coordinatewise
Euler--Skorokhod projection
$\widetilde X_{k+1}=X_k+\mu\Delta t+\Sigma\sqrt{\Delta t}\,\xi_k$,
$X_{k+1}=(\widetilde X_{k+1})^+$, and the local time of $Y$ at zero is estimated
by the occupation-density formula
\begin{equation}
\label{eq:localtime-estimator}
   \widehat L_T^0(Y)
   =\frac{1}{2\varepsilon}\sum_{k=0}^{N-1}
   \mathbf 1_{\{|Y_k|<\varepsilon\}}\,q(X_k)\,\Delta t,
   \qquad \varepsilon=C_\varepsilon\sqrt{\Delta t}\ \ (C_\varepsilon=2).
\end{equation}
Figure~\ref{fig:localtime} shows the resulting estimate of
$\mathbb E[L_t^0(Y)]$. The log--log slope is close to $1/2$ and the ratio
$\widehat{\mathbb E}[L_t^0(Y)]/\sqrt t$ stabilizes near the theoretical value
$2/\sqrt\pi$ (for a driftless Gaussian increment with variance $q=2$,
$\mathbb E[L_t^0(Y)]=\sqrt{4t/\pi}$). Table~\ref{tab:localtime} reports the fitted
exponent $\widehat\beta$ in $\widehat{\mathbb E}[L_t^0(Y)]\approx\widehat C\,t^{\widehat\beta}$
for several time steps; as $\Delta t\downarrow0$ the band bias decreases and
$\widehat\beta\to\tfrac12$. Since the accumulated running cost over $[0,t]$ is
$O(t)$, this confirms the scale separation used in the proof of
Theorem~\ref{thm:no-stopping-kink}: the kink-generated local-time term dominates
the running cost over short horizons.

\begin{figure}[h!]
\centering
\includegraphics[width=\textwidth]{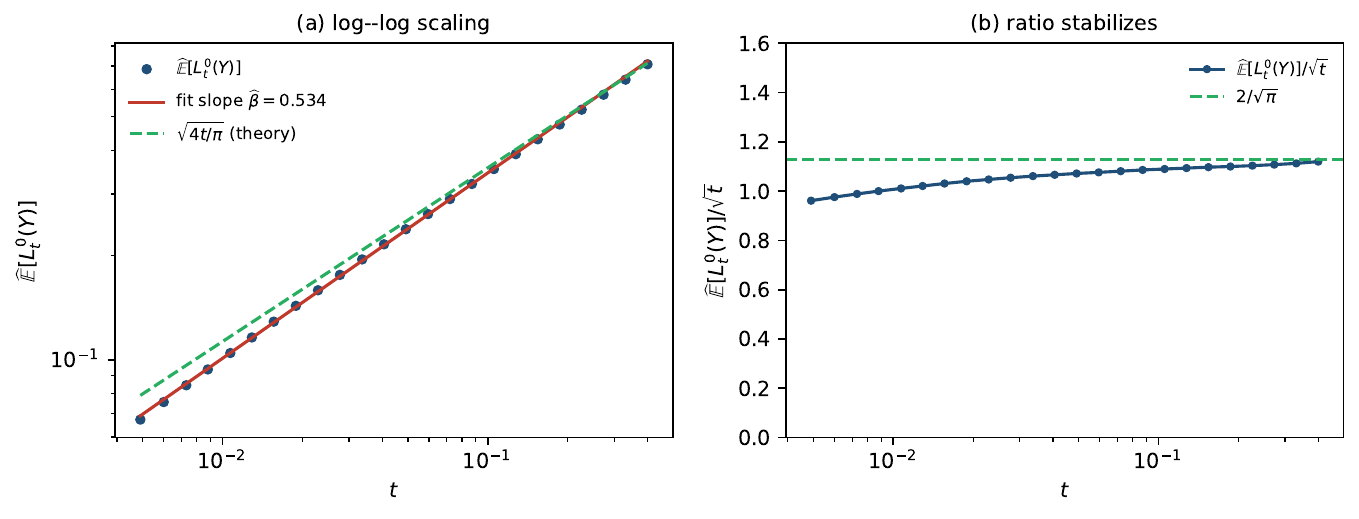}
\caption{Estimated local-time scaling at the diagonal kink, benchmark reflected
Brownian motion started at $x=(1,1)$. (a)~The log--log slope of
$\widehat{\mathbb E}[L_t^0(Y)]$ is close to $1/2$. (b)~The ratio
$\widehat{\mathbb E}[L_t^0(Y)]/\sqrt t$ stabilizes near $2/\sqrt\pi$, confirming
the $\sqrt t$ scaling used in Theorem~\ref{thm:no-stopping-kink}
($M=1.2\times10^5$ paths, $\Delta t=10^{-4}$).}
\label{fig:localtime}
\end{figure}

\begin{table}[h!]
\centering
\caption{Local-time scaling. Fitted exponent $\widehat\beta$ and constant
$\widehat C$ in $\widehat{\mathbb E}[L_t^0(Y)]\approx\widehat C\,t^{\widehat\beta}$
over $t\in[10^{-2},2\times10^{-1}]$, for decreasing time step $\Delta t$ with
band $\varepsilon=2\sqrt{\Delta t}$ and $M=1.2\times10^5$ paths. The theoretical
values are $\beta=\tfrac12$, $C=2/\sqrt\pi\approx1.128$.}
\label{tab:localtime}
\begin{tabular}{ccccc}
\toprule
$\Delta t$ & $\varepsilon$ & $\widehat\beta$ & $\widehat C$ & $M$\\
\midrule
$2.0\times10^{-3}$ & $0.0894$ & $0.625$ & $1.285$ & $1.2\times10^5$\\
$1.0\times10^{-3}$ & $0.0632$ & $0.589$ & $1.241$ & $1.2\times10^5$\\
$5.0\times10^{-4}$ & $0.0447$ & $0.563$ & $1.206$ & $1.2\times10^5$\\
$1.0\times10^{-4}$ & $0.0200$ & $0.534$ & $1.178$ & $1.2\times10^5$\\
\bottomrule
\end{tabular}
\end{table}

\subsection{Killed versus unrestricted resolvents}
\label{sec:num-resolvent-gap}
We verify Proposition~\ref{prop:unrestricted-fails} for one-dimensional reflected
Brownian motion on $[0,\infty)$ with $\mathcal C=[0,b)$, $b=1$, $r=1$, and
$f=\mathbf 1_{[0,b)}$. The killed resolvent is estimated by accumulating
$e^{-rt_k}\Delta t$ along each path until its first passage to $b$, and the
unrestricted resolvent by accumulating $e^{-rt_k}\mathbf 1_{\{X_k<b\}}\Delta t$
over a long horizon $T=8$; first passage is corrected by the Brownian-bridge
crossing probability between grid points to remove the Euler overshoot bias.
Figure~\ref{fig:resolvent} overlays the Monte Carlo estimates on the closed-form
curves of Proposition~\ref{prop:unrestricted-fails}, and Table~\ref{tab:resolvent}
reports the gap and its relative error. The agreement is within one percent at
interior points, so the strict gap is a genuine feature and not a discretization
artifact. Probabilistically, the gap is the discounted occupation accumulated
after the process first reaches the stopping set and then re-enters the
continuation region; this is precisely the occupation an optimal stopping
representation must not count.

\begin{figure}[h!]
\centering
\includegraphics[width=0.8\textwidth]{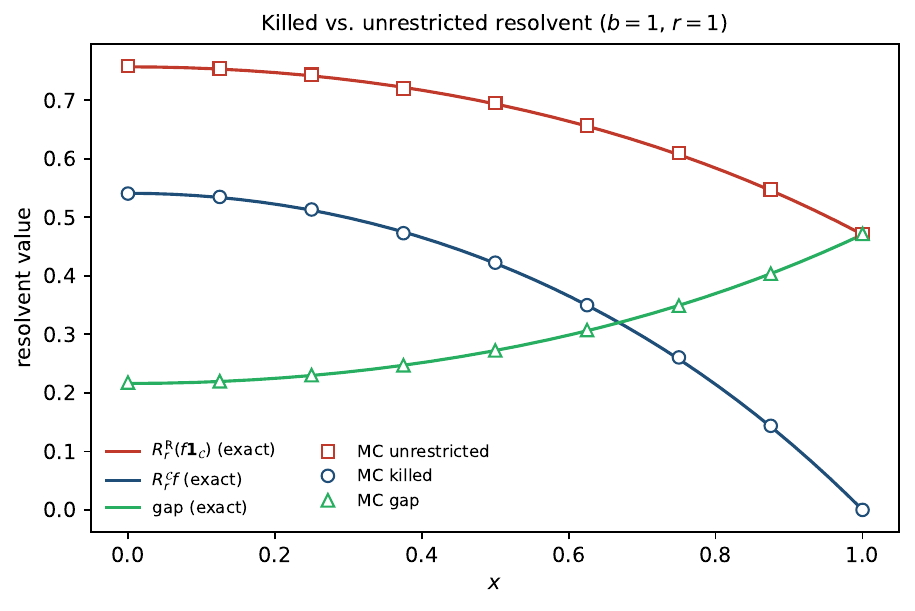}
\caption{Killed and unrestricted resolvents for reflected Brownian motion on the
half-line ($b=1$, $r=1$). Monte Carlo estimates ($M=3\times10^4$ paths per
starting point, $\Delta t=2\times10^{-3}$) agree with the closed-form curves of
Proposition~\ref{prop:unrestricted-fails}; the positive gap is the discounted
post-stopping occupation caused by re-entry into the continuation region.}
\label{fig:resolvent}
\end{figure}

\begin{table}[h!]
\centering
\caption{Resolvent gap
$R_r^{\mathrm R}(f\mathbf 1_{\mathcal C})(x)-R_r^{\mathcal C}f(x)$ for reflected
Brownian motion on $[0,\infty)$, $b=1$, $r=1$: closed form versus Monte Carlo,
with relative error.}
\label{tab:resolvent}
\begin{tabular}{cccc}
\toprule
$x$ & exact gap & Monte Carlo gap & relative error\\
\midrule
$0.00$ & $0.2160$ & $0.2176$ & $0.75\%$\\
$0.25$ & $0.2296$ & $0.2302$ & $0.26\%$\\
$0.50$ & $0.2723$ & $0.2724$ & $0.06\%$\\
$0.75$ & $0.3493$ & $0.3489$ & $0.11\%$\\
$1.00$ & $0.4704$ & $0.4714$ & $0.21\%$\\
\bottomrule
\end{tabular}
\end{table}

\subsection{Diagonal avoidance in the reflected Brownian benchmark}
\label{sec:num-benchmark}
We approximate the reflected obstacle problem \eqref{eq:obstacle-max} for the
benchmark \eqref{eq:benchmark-sde} with $\alpha=1$, $r=0.08$, $c_0=0.05$,
$\mu=(-0.02,-0.02)$, and
$\Sigma=\bigl(\begin{smallmatrix}0.35&0\\0.15&0.30\end{smallmatrix}\bigr)$, so
that $a=\Sigma\Sigma^\top$ has $q=n^\top a\,n=0.115>0$. We use the
Kushner--Dupuis Markov-chain (finite-difference) scheme on $[0,X_{\max}]^2$ with
$X_{\max}=5$, reflecting (mirror) conditions at $x_1=0$ and $x_2=0$, and the
truncation condition $V_h=G$ on the outer edges $x_i=X_{\max}$; the discrete
variational inequality $\max\{(\mathcal L_h-r)V_h-c_0,\;G-V_h\}=0$ is solved by
value iteration. Figure~\ref{fig:benchmark} shows the continuation advantage
$H_h=V_h-G$: it is concentrated as a ridge along the diagonal and is essentially
zero (stopping) elsewhere, so the stopping set fills the bulk of the domain while
avoiding the diagonal. Table~\ref{tab:benchmark} reports the minimum of $H_h$
along the interior diagonal (with a margin of $0.5$ from the boundary) for three
grids; the value is bounded away from zero and stable under refinement,
consistent with Theorem~\ref{thm:no-stopping-kink}. The experiment is not used as
a proof of epigraph structure; it visualises the unconditional diagonal-avoidance
mechanism in a concrete reflected diffusion.

\begin{figure}[h!]
\centering
\includegraphics[width=\textwidth]{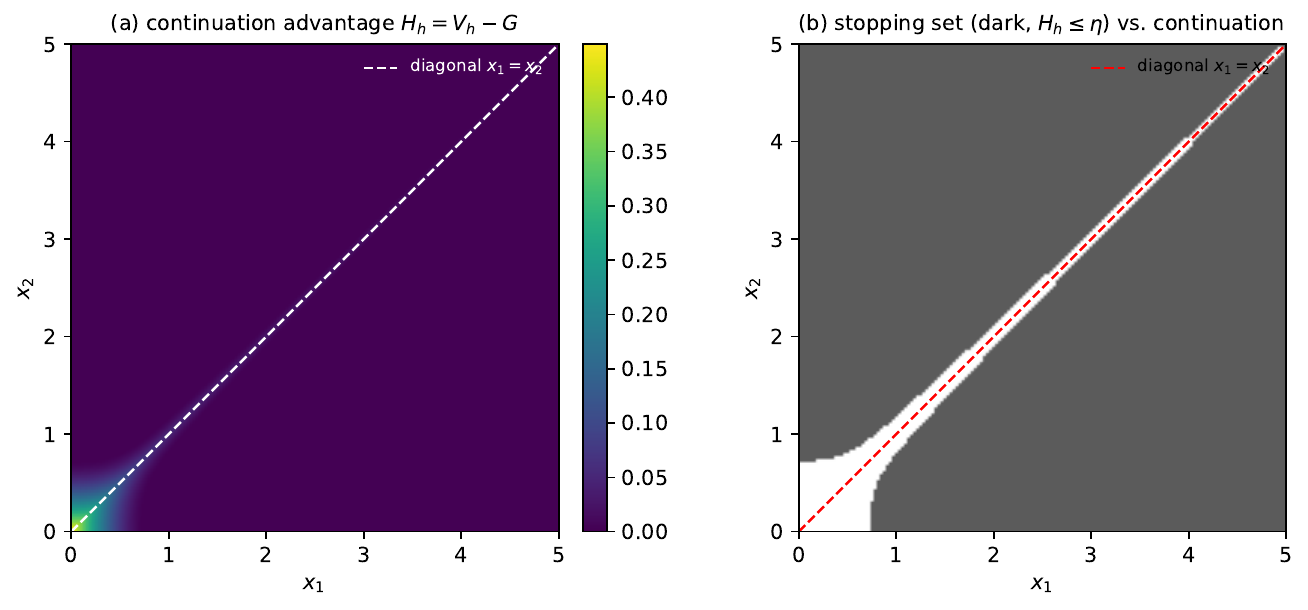}
\caption{Finite-difference approximation of the reflected Brownian benchmark
($N=160$). (a)~The continuation advantage $H_h=V_h-G$ forms a ridge along the
diagonal. (b)~The stopping set $\{H_h\le\eta\}$ (dark, $\eta=10^{-3}$) fills the
bulk of the domain; the diagonal lies inside the thin continuation strip, in
agreement with Theorem~\ref{thm:no-stopping-kink}.}
\label{fig:benchmark}
\end{figure}

\begin{table}[h!]
\centering
\caption{Diagonal avoidance in the benchmark. Minimum and maximum of the
continuation advantage $H_h=V_h-G$ on the interior diagonal
$\{x_1=x_2\}\cap[0.5,X_{\max}-0.5]^2$, and the stopping-set fraction
$\{H_h\le10^{-3}\}$, for three grid sizes. The interior-diagonal minimum is
bounded away from zero and stable under refinement.}
\label{tab:benchmark}
\begin{tabular}{ccccc}
\toprule
$N$ & $\delta$ & $\min_{\mathrm{diag}}H_h$ & $\max_{\mathrm{diag}}H_h$ & stopping fraction\\
\midrule
$80$  & $0.0625$ & $0.0185$ & $0.1264$ & $0.947$\\
$120$ & $0.0417$ & $0.0189$ & $0.1257$ & $0.943$\\
$160$ & $0.0312$ & $0.0184$ & $0.1252$ & $0.943$\\
\bottomrule
\end{tabular}
\end{table}

\subsection{Reproducibility}
\label{sec:num-reproducibility}
All simulations use independent Gaussian increments with fixed random seeds. The
reflected process is generated by the coordinatewise Skorokhod map, local time is
estimated through the occupation-density approximation
\eqref{eq:localtime-estimator}, and first-passage events in the resolvent
experiment use the Brownian-bridge crossing correction. Convergence was checked
by halving both the time step and the occupation bandwidth (Table~\ref{tab:localtime})
and by refining the spatial grid (Table~\ref{tab:benchmark}). The
finite-difference scheme uses the Kushner--Dupuis local-consistency weights, for
which the transition probabilities are non-negative because the covariance is
diagonally dominant ($a_{11}-a_{12}>0$ and $a_{22}-a_{12}>0$). Pseudocode is
given in Appendix~\ref{app:simulation}. The code used to generate the figures is
short and will be made available upon request.

\section{Discussion and limitations}
\label{sec:discussion}
The contributions are cleanly tiered. Unconditionally, within the standing
reflected-diffusion hypotheses, we have shown that the stopping gain of the max
payoff is a signed measure whose diagonal part is the explicit, sign-definite
surface measure of Theorem~\ref{thm:stopping-gain-measure}; that this strict sign
forces the optimal stopper off the kink (Theorem~\ref{thm:no-stopping-kink});
that the value is the killed-resolvent potential of that measure
(Theorem~\ref{thm:killed-resolvent-representation}); and that the unrestricted
reflected resolvent strictly overstates it, by the solved example of
Proposition~\ref{prop:unrestricted-fails}. These statements are verified end to
end in the reflected Brownian benchmark (Theorem~\ref{thm:benchmark}) and
corroborated numerically in Section~\ref{sec:numerics}, where the $\sqrt t$
local-time scaling, the resolvent gap, and the diagonal avoidance are reproduced
directly.

Conditionally, the free-boundary geometry---epigraph representation, monotonicity
of the boundary, smooth fit, and candidate verification---rests on explicit
structural hypotheses on $H=V-G$ that we do not derive from the dynamics. We
regard this honesty as a feature: the multidimensional reflected problem does not
admit a cheap regularity theory, and the literature obtains boundary regularity
only under problem-specific work
\cite{peskir2019continuity,de2019lipschitz,de2015note,de2020global}.

Several questions remain open. First, identifying checkable conditions on
$(\mu,a,c)$ under which Assumption~\ref{ass:vertical-H} holds for
\eqref{eq:benchmark-sde}, and hence the epigraph structure, would upgrade the
benchmark from Tier~I to a fully solved free-boundary problem. Second, the corner
$(0,0)$ was handled by an assumption verified only for non-degenerate covariance;
degenerate or oblique reflection may produce a genuine corner additive functional
that must be added to $\Gamma$. Third, the killed Green kernel
$G_r^{\mathcal C}$ of the reflected diffusion, together with the trace of the
diagonal potential $R_r^{\mathcal C}\Gamma^\Delta$, would turn the trace condition
of Proposition~\ref{prop:boundary-trace} into an explicit integral equation for
the boundary; constructing $G_r^{\mathcal C}$ in the quadrant is itself a
non-trivial potential-theoretic task. Finally, the no-stopping-on-the-kink
phenomenon should extend to general convex piecewise-affine payoffs, with the
diagonal surface measure replaced by a sum of sign-definite surface measures over
the kink set; we expect the local-time scaling argument to carry over verbatim.

\appendix
\section{Simulation details}
\label{app:simulation}
We summarize the three algorithms. Throughout, $\xi_k$ denotes an independent
standard Gaussian vector and $\Delta t$ the time step.

\emph{Reflected Brownian path (benchmark).} Starting from $X_0=x$, iterate for
$k=0,\dots,N-1$:
\[
   \widetilde X_{k+1}=X_k+\mu\,\Delta t+\Sigma\sqrt{\Delta t}\,\xi_k,
   \qquad
   X_{k+1}=(\widetilde X_{k+1})^+\ \text{coordinatewise}.
\]

\emph{Local time at the kink.} With $Y_k=X_k^1-\alpha X_k^2$ and
$\varepsilon=2\sqrt{\Delta t}$, accumulate
\[
   \widehat L_T^0(Y)
   =\frac{1}{2\varepsilon}\sum_{k=0}^{N-1}
   \mathbf 1_{\{|Y_k|<\varepsilon\}}\,q(X_k)\,\Delta t .
\]

\emph{Killed and unrestricted resolvents (half-line).} For each path with
$X_0=x\in[0,b)$, set the killed clock active. At step $k$, add
$e^{-rk\Delta t}\Delta t$ to the unrestricted estimator if $X_k<b$, and to the
killed estimator if additionally the clock is active. After the Euler step to
$X_{k+1}$, deactivate the clock if $X_{k+1}\ge b$, or, when $X_k,X_{k+1}<b$, with
the Brownian-bridge crossing probability
$\exp\!\big(-2(b-X_k)(b-X_{k+1})/(\sigma^2\Delta t)\big)$ ($\sigma^2=1$ here).
Average over paths.

\emph{Finite-difference obstacle solver.} On a uniform grid of spacing $\delta$,
form the Kushner--Dupuis weights
$w_{\pm e_1}=\tfrac12(a_{11}-a_{12})+\delta\mu_1^\pm$,
$w_{\pm e_2}=\tfrac12(a_{22}-a_{12})+\delta\mu_2^\pm$,
$w_{e_1+e_2}=w_{-e_1-e_2}=\tfrac12a_{12}$, with normalization $Q$ equal to their
sum and local time step $\Delta t=\delta^2/Q$. Initialize $V_h=G$ and iterate
$V_h\leftarrow\max\{G,\ -c_0\Delta t+(1+r\Delta t)^{-1}\sum_y p(\cdot,y)V_h(y)\}$,
$p=w/Q$, until the sup-norm increment falls below the tolerance, using mirror
neighbours at $x_i=0$ and $V_h=G$ on the outer edges.

\bibliography{reference}

\end{document}